\documentclass[12pt]{amsart}
\usepackage{latexsym,amsmath,amssymb}
\usepackage{accents}
\usepackage{color}  
\title[Nonlinear commutators for the fractional ${p}$-Laplacian]{Nonlinear commutators for the fractional ${p}$-Laplacian and applications}

\author{Armin Schikorra}

\address{Armin Schikorra, Universit\"at Basel, {\tt armin.schikorra@unibas.ch}}
plus 6pt minus 12pt
\def\eps{\varepsilon}

\def\N{{\mathbb N}}

\def\S{{\mathbb S}}

\newtheorem{theorem}{Theorem}
\newtheorem{lemma}[theorem]{Lemma}

\newtheorem{proposition}[theorem]{Proposition}

\theoremstyle{definition}

\def\diam{{\rm diam\,}}
\def\dist{{\rm dist\,}}

\def\lip{{\rm Lip\,}}

\def\supp{{\rm supp\,}}

\newcommand{\R}{\mathbb{R}}

\newcommand{\Z}{\mathbb{Z}}

\newcommand{\brac}[1]{\left (#1 \right )}

\newcommand{\Sw}{\mathcal{S}}

\newcommand{\barint}{
\rule[.036in]{.12in}{.009in}\kern-.16in \displaystyle\int }

\newcommand{\barcal}{\mbox{$ \rule[.036in]{.11in}{.007in}\kern-.128in\int $}}

\def\mvint_#1{\mathchoice
          {\mathop{\vrule width 6pt height 3 pt depth -2.5pt
                  \kern -8pt \intop}\nolimits_{\kern -3pt #1}}%
          {\mathop{\vrule width 5pt height 3 pt depth -2.6pt
                  \kern -6pt \intop}\nolimits_{#1}}%
          {\mathop{\vrule width 5pt height 3 pt depth -2.6pt
                  \kern -6pt \intop}\nolimits_{#1}}%
          {\mathop{\vrule width 5pt height 3 pt depth -2.6pt
                  \kern -6pt \intop}\nolimits_{#1}}}

\numberwithin{theorem}{section} \numberwithin{equation}{section}

\newcommand{\lap}{\Delta }
\newcommand{\aleq}{\precsim}

\newcommand{\aeq}{\approx}

\newcommand{\laps}[1]{(-\lap) ^{\frac{#1}{2}}}
\newcommand{\lapst}[1]{(-\lap) ^{#1}}
\newcommand{\lapms}[1]{I^{#1}}

\newcommand{\plaps}[3]{(-\lap)^{#1}_{#2,#3}}
\begin{document}

\sloppy

\sloppy


\begin{abstract}
We prove a nonlocal, nonlinear commutator estimate concerning the transfer of derivatives onto testfunctions. 
For the fractional $p$-Laplace operator it implies that solutions to certain degenerate nonlocal equations are higher differentiable. Also, weak fractional $p$-harmonic functions which a priori are less regular than variational solutions are in fact classical.
As an application we show that sequences of uniformly bounded $\frac{n}{s}$-harmonic maps converge strongly outside at most finitely many points.
\end{abstract}

\maketitle
\section{Introduction}
The fractional $p$-Laplacian of order $s \in (0,1)$ on a domain $\Omega \subset \R^n$, $\plaps{s}{p}{\Omega} u$ is a distribution given by
\[
 \plaps{s}{p}{\Omega} u[\varphi] := \int_{\Omega}\int_{\Omega} \frac{|u(x)-u(y)|^{p-2} (u(x)-u(y))\ (\varphi(x)-\varphi(y))}{|x-y|^{n+sp}}\ dx\ dy
\]
for $\varphi \in C_c^\infty(\Omega)$. It appears as the first variation of the $\dot{W}^{s,p}$-Sobolev norm
\[
 [u]_{W^{s,p}(\Omega)}^p := \int_{\Omega}\int_{\Omega} \frac{|u(x)-u(y)|^{p}}{|x-y|^{n+sp}}\ dx\ dy.
\]
In this sense it is related to the classical $p$-Laplacian
\[
 \lap_{p}u = \operatorname{div} (|\nabla u|^{p-2} \nabla u)
\]
which appears as first variation of the $\dot{W}^{1,p}$-Sobolev norm $\|\nabla u\|_{p}^p$.

If $p = 2$ the fractional $p$-Laplacian on $\R^n$ becomes the usual fractional Laplace operator
\[
 \lapst{s} f = \mathcal{F}^{-1}(c\ |\xi|^{2s} \mathcal{F} f),
\]
where $\mathcal{F}$ and $\mathcal{F}^{-1}$ denote the Fourier transform and its inverse, respectively. As a distribution
\[
 \lapst{s} f[\varphi] = \int_{\R^n} \lapst{s} f \varphi.
\]
For an overview on the fractional Laplacian and fractional Sobolev spaces we refer to, e.g., \cite{Hitchhiker,Hitchhiker2}.

Due to the degeneracy for $p \neq 2$, regularity theory for equations involving the $p$-Laplacian is quite delicate, for example $p$-harmonic functions may not be $C^2$. The fractional $p$-Laplacian has recently received quite some interest, for example we refer to  \cite{BjorlandCaffarelliFigalli12,CastroKuusiPalatucciLocalBehaviour,CastroKuusiPalatucciJFA14,SchikorraCPDE14,SireKuusiMingioneSelfImproving,SireKuusiMingioneFracGehring,IannizzottoMosconiSquassina,SireKuusiMingioneMeasureData,VazquezDirichletEvolution}. Higher regularity is one interesting and very challenging question where only very partial results are known, e.g. in \cite{BjorlandCaffarelliFigalli12} they obtain for $s \approx 1$ estimates in $C^{1,\alpha}$.

Our first result is a nonlinear commutator estimate for the fractional $p$-Laplacian. It measures how and at what price one can ``transfer'' derivatives to the testfunction. In the linear case $p=2$ this is just integration by parts: Let $c$ be the constant depending on $s$ and $\eps$ so that $\lapst{s+\eps} = c \lapst{\eps}\circ \lapst{s}$. Then for any testfunction $\varphi$,
\[
 \lapst{s+\eps} u[\varphi] = c\lapst{s} u[\lapst{\eps}\varphi]
\]
In the nonlinear case $p \neq 2$ (we shall restrict our attention for technical simplicity to $p \geq 2$) this is not true anymore. Instead we have 
\begin{theorem}\label{th:epsest}
Let $s \in (0,1)$, $p \in [2,\infty)$ and $\eps \in [0,1-s)$. Take $B \subset \R^n$ a ball or all of $\R^n$. Let $u \in W^{s,p}(B)$ and $\varphi \in C_c^\infty(B)$. For a certain constant $c$ depending on $s, \eps, p$, denote the nonlinear commutator
\[
 R(u,\varphi) := \plaps{s+\eps}{p}{B}u[\varphi] - c\plaps{s}{p}{B}u[\laps{\eps p}\varphi].
\]
Then we have the estimate
\[
 |R(u,\varphi)| \leq C\ \eps\ [u]_{W^{s+\eps,p}(B)}^{p-1} [\varphi]_{W^{s+\eps,p}(\R^n)}.
\]
\end{theorem}
The fact that the $\eps$ appears in the estimate of $R(u,\varphi)$ is the main point in Theorem~\ref{th:epsest}. It relies  on a logarithmic potential estimate:
\begin{lemma}\label{la:logest}
Let for $\alpha, \beta \in (0,n)$,
\[
k(x,y,z) =\brac{|x-z|^{\alpha-n}\log\frac{|x-z|}{|x-y|}-|y-z|^{\alpha-n}\log\frac{|y-z|}{|x-y|}}.
\]
Let $\gamma \in (0,1)$, $p \in (1,\infty)$ and assume that $s:= \gamma + \beta - \alpha \in (0,1)$. We consider the following semi-norm expression for $\varphi \in C_c^\infty(\R^n)$
\[
   A(\varphi) := \brac{\int_{\R^n} \int_{\R^n}\left |\int_{\R^n} k(x,y,z)\ \laps{\beta} \varphi(z)\ dz \right |^p\ \frac{dx\ dy}{|x-y|^{n+\gamma p}}}^{\frac{1}{p}}.
\]
We have \[A(\varphi) \leq C [\varphi]_{W^{s,p}(\R^n)}.\]
\end{lemma}

The additional factor $\eps$ in Theorem~\ref{th:epsest} facilitates estimates ``close to the differential order $s$''. More precisely 

\begin{theorem}\label{th:est1}
Let $s \in (0,1)$, $p \in [2,\infty)$, and a domain $\Omega \subset \R^n$, and $u \in W^{s,p}(\Omega)$ be a solution to $\plaps{s}{p}{\Omega} u = f$, i.e.
\[
 \int_{\Omega}\int_{\Omega} \frac{|u(x)-u(y)|^{p-2}(u(x)-u(y))\ (\varphi(x)-\varphi(y))}{|x-y|^{n+sp}}\ dx\ dy = f[\varphi]
\]
for all $\varphi \in C_c^\infty(\Omega)$.
Then there is an $\eps_0 > 0$ only depending on $s$, $p$, and $\Omega$, so that for $\eps \in (0,\eps_0)$ the following holds: If $f \in (W^{s-\eps(p-1),p}(\Omega))^\ast$ then $u \in W^{s+\eps,p}_{loc}(\Omega)$.

More precisely, we have for any $\Omega_1 \Subset \Omega$ a constant $C= C(\Omega_1,\Omega,s,p)$ so that
\[
[u]_{W^{s+\eps,p}(\Omega_1)} \leq C\ \|f\|_{(W_0^{s-\eps(p-1),p}(\Omega))^\ast} + C[u]_{W^{s,p}(\Omega)}.
\]

Also, by Sobolev imbedding, the higher differentiability $W^{s+\eps,p}_{loc}$ implies higher integrability i.e. $W^{s,p+\frac{pn}{n-\eps p}}_{loc}$-estimates.
\end{theorem}

In the regime $p=2$, a higher differentiability result similar to Theorem~\ref{th:est1} was proven by Kuusi, Mingione, and Sire \cite{SireKuusiMingioneSelfImproving}. It seems also possible to extend their approach to the case $p > 2$. Their argument is based on a generalization of Gehring's Lemma and it is also valid for nonlinear versions, see \cite{SireKuusiMingioneFracGehring}. Our method is similarly robust. Indeed one can show
\begin{theorem}\label{th:estv2}
Let $s \in (0,1)$, $p \in [2,\infty)$, and a domain $\Omega \subset \R^n$. Let $\phi: \R \to \R$ and $K(x,y)$ be a measurable kernel so that for some $C > 1$,
\[
 |\phi(t)|\leq C |t|^{p-1}, \quad \phi(t)t \geq |t|^{p} \quad \forall t \in \R,
\]
and
\[
 C^{-1} |x-y|^{-n-sp} \leq K(x,y) \leq C |x-y|^{-n-sp}.
\]
We consider for $u \in W^{s,p}(\Omega)$, the distribution $\mathcal{L}_{\phi,K,\Omega}(u)$
\[
 \mathcal{L}_{\phi,K,\Omega}(u)[\varphi] := \int_{\Omega}\int_{\Omega} K(x,y)\ \phi(u(x)-u(y))\ (\varphi(x)-\varphi(y))\ dx\ dy
\]
Then the conclusions of Theorem~\ref{th:est1} still hold if the fractional $p$-Laplace $\plaps{s}{p}{\Omega}$ is replaced with $\mathcal{L}_{\phi,K,\Omega}$. \end{theorem}
Since the arguments for Theorem~\ref{th:estv2} follow closely the proof of Theorem~\ref{th:est1}, we leave this as an exercise to the interested reader.

There is also a reminiscent result to Theorem~\ref{th:est1} the usual $p$-Laplace: A nonlinear potential estimate due to Iwaniec \cite{Iw92}. It implies that for $u$ with $\supp u \subset \Omega$ there are maps $v$, $R$, so that
\[
 |\nabla u|^{\eps} \nabla u = \nabla v + R,
\]
with $\|\nabla v\|_{\frac{q}{1+\eps},\Omega} \aleq \|\nabla u\|^{1+\eps}_{q,\Omega}$ for all $q$ and 
\[
 \|R\|_{\frac{p+\eps}{1+\eps},\Omega} \aleq \eps \|\nabla u\|^{1+\eps}_{p+\eps,\Omega}.
\]
In this situation, the additional $\eps$ in the last estimate allows for estimates ``close to the \emph{integrability} order $p$''. Indeed
\[
 \|\nabla u\|^{p+\eps}_{p+\eps, \Omega} = \int_{\Omega} |\nabla u|^{p-2} \nabla u \nabla v + \int_{\Omega} |\nabla u|^{p-2} \nabla u R,
\]
and thus,
\[
 \|\nabla u\|_{p+\eps, \Omega}^{p+\eps} \aleq |\lap_p u[v]| + \eps \|\nabla u\|_{p+\eps,\Omega}^{p-1}\ \|\nabla u\|^{1+\eps}_{p+\eps,\Omega}.
\]
In particular, if $\eps$ is small enough and $\lap_p u$ is in $(W_0^{1,\frac{p+\eps}{1+\eps}}(\Omega))^\ast$, then $u \in W^{1,p+\eps}(\Omega)$.

The commutator estimate in Theorem~\ref{th:epsest} also allows to estimate very weak solutions - i.e. solutions whose initial regularity assumptions are below the variationally natural regularity:

In the local regime, the distributional $p$-Laplacian $\lap_p u[\varphi]$ is well defined for $\varphi \in C_c^\infty(\Omega)$ whenever $u \in W^{1,p-1}_{loc}(\Omega)$. The variationally natural regularity assumption is however $W^{1,p}$, since $\lap_p$ appears as first variation of $\|\nabla u \|_{p,\Omega}^p$. For the $p$-Laplacian, Iwaniec and Sbordone \cite{IwaniecSbordone94} showed that some weak $p$-harmonic functions are in fact classical variational solutions:
\begin{theorem}[Iwaniec-Sbordone]
For any $p \in (1,\infty)$, $\Omega \subset \R^n$, there are exponents $1 < r_1 < p < r_2 < \infty$ so that every (weakly) $p$-harmonic map,
\[
 \lap_p u = 0,
\]
satisfying $u \in W^{1,r_1}_{loc}(\Omega)$ indeed belongs to $W^{1,r_2}_{loc}(\Omega)$.
\end{theorem}
Again, while the $p$-Laplace improves its solution's \emph{integrability}, the fractional $p$-Laplace improves its solution's \emph{differentiability}. The distributional fractional $p$-Laplace $\plaps{s}{p}{\Omega}u[\varphi]$ is well defined for $\varphi \in C_c^\infty(\Omega)$ whenever $u \in W^{q,p-1}(\Omega)$ for any $q > 0$ with $q \geq (\frac{sp-1}{p-1})_+$. We have
\begin{theorem}\label{th:veryweak}
For any $s \in (0,1)$ $p \in (2,\infty)$, $\Omega \subset \R^n$, there are exponents $1 < r_1 < p < r_2 < \infty$ and $t_1 < s < t_2$ so that every (weakly) $s$-$p$-harmonic map,
\[
 \plaps{s}{p}{\Omega} u = 0,
\]
satisfying $u \in W^{t_1,r_1}(\Omega)$ indeed belongs to $W^{t_2,r_2}_{loc}(\Omega)$.
\end{theorem}
The arguments for Theorem~\ref{th:veryweak} are quite similar to the ones in Theorem~\ref{th:est1}, and we shall skip them.

Let us state an important application of Theorem~\ref{th:est1}: It is concerning degenerate fractional harmonic maps into spheres $\S^N \subset \R^{N+1}$: In \cite{SchikorraCPDE14} we proved that for $s \in (0,1)$ critical points of the energy
\[
 \mathcal{E}_{s} (u) := \int_{\Omega} \int_{\Omega} \frac{|u(x)-u(y)|^{\frac{n}{s}}}{|x-y|^{n+s \frac{n}{s}}}\ dx\ dy, \quad u: \Omega \subset \R^n \to \S^N
\]
are H\"older continuous. Indeed, together with Theorem~\ref{th:est1} the estimates in \cite{SchikorraCPDE14} impAly a sharper result
\begin{theorem}[$\eps$-regularity for fractional harmonic maps]\label{th:higherreg}
For any open set $\Omega \subset \R^n$ there is a $\delta > 0$ so that for any $\Lambda > 0$ there exists $\eps > 0$ and the following holds: Let $u \in W^{s,\frac{n}{s}}(\Omega,\S^N)$ with
\begin{equation}\label{eq:uwsns}
 [u]_{W^{s,\frac{n}{s}}(\Omega)} \leq \Lambda
\end{equation}
be a critical point of $\mathcal{E}_s(u)$, i.e.
\begin{equation}\label{eq:criticalpt}
 \frac{d}{dt}\Big |_{t = 0} \mathcal{E}_s\brac{\frac{u+t\varphi}{|u+t\varphi|}} = 0 \quad \forall \varphi \in C_c^\infty(\Omega,\R^N).
\end{equation}
If on a ball $2B \subset \Omega$ we have
\begin{equation}\label{eq:smalluloc}
 [u]_{W^{s,\frac{n}{s}}(2B)} \leq \eps,
\end{equation}
then on the ball $B$ (the ball concentric to $2B$ with half the radius),
\[
 [u]_{W^{s+\delta,\frac{n}{s}}(B)} \leq C_{\Lambda,B}.
\]
\end{theorem}
This kind of $\eps$-regularity estimate is crucial for compactness and bubble analysis for fractional harmonic maps. Da Lio obtained quantization results \cite{DaLioBubbles} for $n=1$ and $s=\frac{1}{2}$. With the help of Theorem~\ref{th:higherreg} one can extend her \emph{compactness} estimates to all $s \in (0,1)$, $n \in \N$. More precisely, we have the following result extending the first part of \cite[Theorem~1.1]{DaLioBubbles}.
\begin{theorem}\label{th:aecompactness}
Let $u_k \in \dot{W}^{s,\frac{n}{s}}(\R^n,\S^{N-1})$ be a sequence of $(s,\frac{n}{s})$-harmonic maps in the sense of \eqref{eq:criticalpt} such that
\[
 [u_k]_{W^{s,\frac{n}{s}}(\R^n,\S^{N-1})} \leq C.
\]
Then there is $u_\infty \in \dot{W}^{s,\frac{n}{s}}(\R^n,\S^{N-1})$ and a possibly empty set $\{\alpha_1,\ldots,\alpha_l\}$ such that up to a subsequence we have strong convergence away from $\{\alpha_1,\ldots,\alpha_l\}$, that is
\[
 u_k \xrightarrow{k \to \infty} u_\infty \quad \mbox{in $W^{s,\frac{n}{s}}_{loc}(\R^n \backslash \{\alpha_1,\ldots,\alpha_l\})$}.
\]
\end{theorem}
A more precise analysis of compactness and the formation of bubbles will be part of a future work.

\section{Outline and Notation}
In Section~\ref{s:thepsest} we will prove the commutator estimate, Theorem~\ref{th:epsest}. Roughly speaking, we compute the kernel $\kappa_\eps(x,y,z)$ of the commutator and show that its derivative \emph{in $\eps$} (which gives a logarithmic potential) induces a bounded operator. The latter estimate is contained in Lemma~\ref{la:logest} which we shall prove via Littlewood-Paley theory in Section~\ref{s:la:logest}. 

Based on Theorem~\ref{th:epsest} we will then proceed in Section~\ref{s:th:est1} with the proof of Theorem~\ref{th:est1}. Finally, the consequences of this analysis, i.e. higher differentiability result for $p$-fractional harmonic maps is sketched in Section~\ref{s:th:higherreg}, and the proof of Theorem~\ref{th:aecompactness} in Session~\ref{s:th:aecompactness}. In the appendix we record a few necessary tools used throughout the proofs.

We try to keep the notation as simple as possible. For a ball $B$, $\lambda B$ denotes the concentric ball with $\lambda$-times the radius. With
\[
 (u)_B := |B|^{-1} \int_B u
\]
we denote the mean value.

The dual norm of the $p$-Laplacian is denoted as
\[
  \|\plaps{s}{p}{\Omega} u\|_{(W^{t,p}_0(\Omega))^\ast} \equiv \sup_{\varphi} |\plaps{s}{p}{\Omega} u [\varphi]|
\]
where the supremum is taken over $\varphi \in C_c^\infty(\Omega)$ with $[\varphi]_{W^{t,p}(\R^n)} \leq 1$.

We already defined the fractional Laplacian $\laps{s}$. Its inverse $\lapms{s}$ is the Riesz potential, which for some constant $c \in \R$ can be written as
\begin{equation}\label{eq:riespot}
 \lapms{s} g(x) = c\int_{\R^n} |x-z|^{s - n} g(z)\ dz.
\end{equation}

In the estimates, the constants can change from line to line. Whenever we deem the constant unimportant to the argument, we will drop it, writing $A \aleq B$ if $A \leq C\cdot B$ for some constant $C > 0$. Similarly we will use $A \aeq B$ whenever $A$ and $B$ are comparable.

\section{The commutator estimate: Proof of Theorem~\ref{th:epsest}}\label{s:thepsest}
\begin{proof}
Recall that for $t \in (0,n)$ there is a constant $c \in \R$ so that for any $\varphi \in C_c^\infty(\R^n)$,
\begin{equation}\label{eq:rieszpot}
 c \int_{\R^n} |x-z|^{t-n} \laps{t} \varphi(z)\ dz = \lapms{t} \laps{t} \varphi(x) = \varphi(x).
\end{equation}
We write 
\[\plaps{s+\eps}{p}{B}u[\varphi]
=\int\limits_B \int\limits_B \frac{|u(x)-u(y)|^{p-2}(u(x)-u(y))(\frac{\varphi(x)-\varphi(y)}{|x-y|^{\eps p}})}{|x-y|^{n+sp}} \ dx\ dy
\]
\[\overset{\eqref{eq:rieszpot}}{=}\int\limits_{\R^n} \int\limits_B \int\limits_B \frac{|u(x)-u(y)|^{p-2}(u(x)-u(y))(\frac{|x-z|^{t+\eps p-n}-|y-z|^{t+\eps p-n}}{|x-y|^{\eps p}})}{|x-y|^{n+sp}} \ dx\ dy\ \laps{t+\eps p} \varphi(z)dz\]
\[=\int\limits_{\R^n} \int\limits_B \int\limits_B \frac{|u(x)-u(y)|^{p-2}(u(x)-u(y))\brac{|x-z|^{t-n} - |x-y|^{t-n})}}{|x-y|^{n+sp}} \ dx\ dy\ \laps{t+\eps p} \varphi(z)dz\]
\[+\int\limits_{\R^n} \int\limits_B \int\limits_B \frac{|u(x)-u(y)|^{p-2}(u(x)-u(y)) \kappa_\eps(x,y,z) }{|x-y|^{n+sp}} \ dx\ dy\ \laps{t+\eps p} \varphi(z)dz\\
\]
with
\[
 \kappa_\eps(x,y,z) := \brac{ \frac{|x-z|^{t+\eps p-n}-|y-z|^{t+\eps p-n}}{|x-y|^{\eps p}}} -(|x-z|^{t-n} - |x-y|^{t-n}).
\]
Using again \eqref{eq:rieszpot} this reads as
\[R(u,\varphi) := \plaps{s+\eps}{p}{B}u[\varphi]-c\plaps{s}{p}{B}u[\laps{\eps p}\varphi]  \]
\[=\int\limits_{\R^n} \int\limits_B \int\limits_B \frac{|u(x)-u(y)|^{p-2}(u(x)-u(y)) \kappa_\eps(x,y,z) }{|x-y|^{n+sp}} \ dx\ dy\ \laps{t+\eps p} \varphi(z)dz.\\
\]
Since $\kappa_0(x,y,z) = 0$ for almost all $x,y,z \in \R^n$,
\[
 \kappa_\eps(x,y,z) = \int_0^\eps \frac{d}{d\delta}\kappa_\delta(x,y,z)\ d\delta.
\]
We thus set
\[
\begin{split}
 k_\delta(x,y,z) &:= |x-y|^{\delta p} \frac{d}{d\delta} \kappa_\delta(x,y,z) \\
 &=\brac{|x-z|^{t+\delta p-n}\log\frac{|x-z|}{|x-y|}-|y-z|^{t+\delta p-n}\log\frac{|y-z|}{|x-y|}}
 \end{split}
\]
and arrive at $R(u,\varphi)$ being equal to
\[
\int\limits_0^\eps \int\limits _B \int\limits _B \frac{|u(x)-u(y)|^{p-2}(u(x)-u(y))}{|x-y|^{(s+\eps)(p-1)}} \brac{\int\limits_{\R^n} \frac{\kappa_\delta(x,y,z)\ \laps{t+\eps p} \varphi(z)dz}{|x-y|^{s+\eps-(\eps-\delta) p}}} \ \frac{dx\ dy\ d\delta}{|x-y|^n}.\\
\]
With H\"older inequality we get the upper bound for $|R(u,\varphi)|$  
\[\eps [u]_{W^{s+\eps,p}(B)}^{p-1}\ \sup_{\delta \in (0,\eps)}  \brac{\int\limits_B \int\limits_B \brac{\int\limits_{\R^n} \frac{\kappa_\delta(x,y,z)\  \laps{t+\eps p} \varphi(z)dz}{|x-y|^{s+\eps -(\eps -\delta) p}}}^p \frac{dx\ dy}{|x-y|^n}}^{\frac{1}{p}} .\\
\]
This falls into the realm of Lemma~\ref{la:logest}, for
\[
 \alpha := t+\delta p, \quad \beta := t+\eps p, \quad \gamma := s+\eps -(\eps -\delta) p, \quad \gamma+\beta -\alpha = s+\eps.
\]
This concludes the proof.
\end{proof}

\section{Logarithmic potential estimate: Proof of Lemma~\ref{la:logest}}\label{s:la:logest}
For the proof of Lemma~\ref{la:logest} we will use the Littlewood-Paley decomposition: We refer to the Triebel monographs, e.g. \cite{Triebel1983} and \cite{GrafakosMF} for a complete picture of this tool. We will only need few properties:

For a tempered distribution $f$ we define $f_j$ to be the Littlewood-Paley projections $f_j := P_j f$, where
\[
 P_j f(x) := \int \limits_{\R^n} 2^{jn} p(2^{j} (x-z)) f(z)\ dz.
\]
Here, $p$ is a Schwartz function, and it can be chosen in a way such that
\begin{equation}\label{eq:sumfjf}
 \sum_{j \in \Z} f_j = f \quad \mbox{for all $f \in \Sw'$}.
\end{equation}
For any $j \in \Z$ we have the estimate for Riesz potentials and derivatives (cf. \eqref{eq:riespot})
\begin{equation}\label{eq:lapmslapstfj}
\|\lapms{s}|\laps{t} f_j| \|_{p} \aleq \sum_{i = j-1}^{j+1} 2^{j(t-s)} \|f_i \|_{p} 
\end{equation}
The homogeneous semi-norm for the Triebel space $\dot{F}^s_{p,p} = \dot{B}^s_{p,p}$ is
\begin{equation}\label{eq:triebelnorm}
 \|f\|_{\dot{F}^s_{p,p}} := \brac{\sum_{j \in \Z} 2^{jsp} \|f_j\|_{p}^p}^{\frac{1}{p}}.
\end{equation}
Crucially to us, the Triebel spaces are equivalent to Sobolev spaces: For $s \in (0,1)$ we have the identification
\begin{equation}\label{eq:trieblsob}
 \|f\|_{\dot{F}^s_{p,p}} \approx [f]_{W^{s,p}(\R^n)}.
\end{equation}

\begin{proof}[Proof of Lemma~\ref{la:logest}]
For $k \in \Z$, we use the annular cutoff function
\[
 \chi_{|y| \approx 2^{-k}} := \chi_{B_{2^{-k}(0)}\backslash B_{2^{-k-1}(0)}}(y).
\]
With this and \eqref{eq:sumfjf}, setting
\[
 T\varphi(x,y) := \int\limits_{\R^n} k(x,y,z)\ \laps{\beta}
\varphi(z)\ dz,
\]
we decompose
\[
  A(\varphi)^p
 \aleq \sum_{k\in \Z,j \in \Z} I_{j,k},
\]
where
\[
 I_{j,k} := \int\limits_{\R^n} \int\limits_{\R^n} \chi_{|x-y| \approx 2^{-k}}\left |T\varphi(x,y) \right |^{p-1}\ |T\varphi_j(x,y)| \ \frac{dx\ dy}{|x-y|^{n+\gamma p}}.
\]
Set
\[
 a_k := \brac{\int\limits_{\R^n} \int\limits_{\R^n} \chi_{|x-y| \approx 2^{-k}}\left |T\varphi(x,y)\right |^{p} \frac{dx\ dy}{|x-y|^{n+\gamma p}}}^{\frac{1}{p}}
\]
and
\[
 b_j := 2^{j(\gamma+\beta-\alpha)}\ \|\varphi_j\|_{p}.
\]
Note that with \eqref{eq:triebelnorm} and \eqref{eq:trieblsob}
\begin{equation}\label{eq:sumbjak}
 \brac{\sum_{k \in \Z} a_k^p}^{\frac{1}{p}} \aeq A(\varphi) \quad \mbox{and}\quad  \brac{\sum_{j \in \Z} b_j^p}^{\frac{1}{p}} \aeq \|\varphi\|_{\dot{F}^s_{p,p}} \aeq [\varphi]_{W^{s,p}(\R^n)}.
\end{equation}

Then with H\"older inequality,
\[
\begin{split}
 I_{j,k} 
\aleq\ & a_k^{p-1}\ \brac{\int\limits_{\R^n}\int\limits_{\R^n} \chi_{|x-y| \approx 2^{-k}}\left |T\varphi_j(x,y) \right |^{p}\ \ \frac{dx\ dy}{|x-y|^{n+\gamma p}}}^{\frac{1}{p}}\\
=: &a_k^{p-1}\ \tilde{I}_{j,k}.
 \end{split}
 \]
Now we have to possibilities of estimating $\tilde{I}_{j,k}$: 
 
\emph{Firstly}, for any small $\sigma \in (0,\alpha)$ we can employ the estimate $|\log \frac{|x-z|}{|x-y|}| \aleq \frac{|x-y|^\sigma}{|x-z|^\sigma} + \frac{|x-z|^\sigma}{|x-y|^\sigma}$, and have an estimate with Riesz potentials \eqref{eq:riespot}
\[
\begin{split}
 &\int\limits_{\R^n} |x-z|^{\alpha-n}\log\frac{|x-z|}{|x-y|} |\laps{\beta}\varphi_j(z)|\ dz\\
 \aleq &|x-y|^{-\sigma} \lapms{\alpha+\sigma} |\laps{\beta} \varphi_j|(x) + |x-y|^{\sigma} \lapms{\alpha-\sigma} |\laps{\beta} \varphi_j|(x).
\end{split}
 \]
Having in mind \eqref{eq:lapmslapstfj} we obtain the estimate
\[
 \begin{split}
\tilde{I}_{j,k} \aleq  &2^{k(\frac{n+\gamma p}{p})}\ 2^{k\sigma}  2^{-k\frac{n}{p}} \|\lapms{\alpha+\sigma} |\laps{\beta} \varphi_j|\|_{p} + 2^{k(\frac{n+\gamma p}{p})}\ 2^{-k\sigma}  2^{-k\frac{n}{p}} \|\lapms{\alpha-\sigma} |\laps{\beta} \varphi_j|\|_{p}\\
\aleq &2^{(k-j)(\gamma+\sigma)} (b_{j-1} + b_j + b_{j+1}) + 2^{(k-j)(\gamma-\sigma)} (b_{j-1} + b_j + b_{j+1}).
 \end{split}
\]
This is our first estimate:
\begin{equation}\label{eq:bruteest}
 \tilde{I}_{j,k} \aleq 2^{(k-j)(\gamma-\sigma)}\ (2^{2\sigma(k-j)}+ 1)\ (b_{j-1} + b_j + b_{j+1}).
\end{equation}
\emph{Secondly}, by a substitution we can write
\[
 T\varphi_j (x,y) = \int_{\R^n} |z|^{\alpha-n}\log\frac{|z|}{|x-y|} \brac{\laps{\beta}\varphi_j(z+x) -\laps{\beta}\varphi_j(z+y)} dz.
\]
We use now $|f(x) - f(y)| \aleq |x-y| (\mathcal{M}|\nabla f|(x)+\mathcal{M}|\nabla f|(y))$, where $\mathcal{M}$ is the Hardy-Littlewood maximal function. Then, again for any $\sigma > 0$,
\[
 \begin{split}
&|T\varphi_j(x,y)| \\
\aleq &|x-y| \int_{\R^n} |z|^{\alpha-n}\left |\log\frac{|z|}{|x-y|} \right|\ \mathcal{M}|\laps{\beta}\nabla \varphi_j|(z+x)\ dz \\
&+|x-y| \int_{\R^n} |z|^{\alpha-n} \left |\log\frac{|z|}{|x-y|} \right|\ |\mathcal{M}\laps{\beta}\nabla \varphi_j|(z+x)\ dz \\
\aleq &|x-y|^{1-\sigma} \lapms{\alpha +\sigma} \mathcal{M}|\laps{\beta}\nabla \varphi_j|(x)\\
&+ |x-y|^{1-\sigma} \lapms{\alpha +\sigma} \mathcal{M}|\laps{\beta}\nabla \varphi_j|(y)\\
&+ |x-y|^{1+\sigma} \lapms{\alpha -\sigma} \mathcal{M}|\laps{\beta}\nabla \varphi_j|(x)\\
&+ |x-y|^{1+\sigma} \lapms{\alpha -\sigma} \mathcal{M}|\laps{\beta}\nabla \varphi_j|(y)\\
 \end{split}
\]
Consequently, our second estimate is
\[
 \begin{split}
\tilde{I}_{j,k} \aleq  & 2^{k(\gamma -1+\sigma)}  \|\lapms{\alpha +\sigma} \mathcal{M}|\laps{\beta}\nabla \varphi_j\|_{p} + 2^{k(\gamma -1-\sigma)}  \|\lapms{\alpha -\sigma} \mathcal{M}|\laps{\beta}\nabla \varphi_j\|_{p} \\
\aleq & 2^{k(\gamma -1+\sigma)} \ 2^{j(-\alpha-\sigma+\beta+1)} \|\varphi_j\|_{p} + 2^{k(\gamma -1-\sigma)}\ 2^{j(-\alpha+\sigma+\beta+1)}  \|\varphi_j\|_{p}.
\end{split}
\]
Together with \eqref{eq:bruteest} we thus have
\[
 \tilde{I}_{k,j} \aleq \min \{2^{(k-j)(\gamma-\sigma)}\ (2^{2\sigma(k-j)}+ 1), 2^{(j-k)(1-\gamma -\sigma)}\ (1+ 2^{(j-k)(2\sigma)}) \}\ (b_{j-1}+ b_{j} + b_{j+1}) .
\]
In particular, since $\gamma \in (0,1)$ pick any $0 < \sigma < \min \{\gamma,1-\gamma\}$ -- which, as we shall see in a moment, makes the following sums convergent:
\[
\begin{split}
 A(\varphi)^p \aleq &\sum_{j\in \Z} \sum_{k = j+1}^\infty   2^{(j-k)(1-\gamma -\sigma)}\ (b_{j-1}+ b_{j} + b_{j+1})\ a_j^{p-1}\\
 &+\sum_{j\in \Z} \sum_{k = -\infty }^{j-1} 2^{(k-j)(\gamma-\sigma)}\  (b_{j-1}+ b_{j} + b_{j+1})\ a_j^{p-1}\\
 &+\sum_{j\in \Z} (b_{j-1}+ b_{j} + b_{j+1})\ a_j^{p-1}\\
=:& I + II + III.
\end{split}
\]
With H\"older inequality and \eqref{eq:sumbjak},
\[
 III \aleq (\sum_{j \in \Z} b_j^p)^{\frac{1}{p}}\ (\sum_{j \in \Z} a_j^p)^{\frac{p-1}{p}} = A(\varphi)^{p-1}\ [\varphi]_{W^{s,p}(\R^n)}.
\]

As for $I$, for any $\eps > 0$,
\[
\begin{split}
 I =& \sum_{j\in \Z} \sum_{k = j}^\infty 2^{(j-k)(1-\gamma-\sigma)}\ 
b_j\  a_k^{p-1}\\
\aleq & \sum_{j\in \Z} \sum_{k = j}^\infty 2^{(j-k)(1-\gamma-\sigma)}\ 
(\eps^p b_j^p\  + \eps^{-p'} a_k^{p})\\
=&  C_{1-\gamma-\sigma} \eps^p \sum_{j\in \Z} b_j^p\  + \eps^{-p'} \sum_{j\in \Z} \sum_{k = j}^\infty 2^{(j-k)(1-\gamma-\sigma)} a_k^{p}\\
=&  C_{1-\gamma-\sigma} \eps^p \sum_{j\in \Z} b_j^p\  + \eps^{-p'}\sum_{k \in \Z} \sum_{j=-\infty}^{k}  2^{(j-k)(1-\gamma-\sigma)} a_k^{p}\\
=&  C_{1-\gamma-\sigma} \eps^p \sum_{j\in \Z} b_j^p\  + \eps^{-p'}C_{1-\gamma-\sigma} \sum_{k \in \Z} a_k^{p}\\
\aeq &  \eps^p [\varphi]_{W^{s,p}(\R^n)}^p\  + \eps^{-p'}C_{1-\gamma-\sigma} A(\varphi)^p\\
\end{split}
 \]
The same works for $II$: 
\[
\begin{split}
 II =& \sum_{j\in \Z} \sum_{k = -\infty}^{j-1} 2^{(k-j)(\gamma-\sigma)}\ 
b_j\  a_k^{p-1}\\
\aleq &  \eps^p [\varphi]_{W^{s,p}(\R^n)}^p\  + \eps^{-p'}C_{1-\gamma-\sigma} A(\varphi)^p\\
\end{split}
 \]
Together,
\[
 I + II \aleq  \eps^p [\varphi]_{W^{s,p}(\R^n)}^p\  + \eps^{-p'}C_{1-\gamma-\sigma} A(\varphi)^p,
\]
which holds for any $\eps > 0$. Pick \[
          \eps := [\varphi]_{W^{s,p}(\R^n)}^{-\frac{1}{p' }}\ A(\varphi)^{\frac{1}{p' }}. 
        \]
Then
\[
 A(\varphi)^p \leq I+II+III \aleq  A(\varphi)^{p-1}\ [\varphi]_{W^{s,p}(\R^n)}.
\]
We conclude dividing both sides by $A(\varphi)^{p-1}$.
\end{proof}

\section{Higher Differentiability: Proof of Theorem~\ref{th:est1}}\label{s:th:est1}
In view of Lemma~\ref{la:extension} we can assume w.l.o.g. that $\Omega$ is a bounded open set, and that the support of $u$ is strictly contained in some open set $\Omega_1 \Subset \Omega$. Then Theorem~\ref{th:est1} follows from	
\begin{lemma}
Let $\Omega_1 \Subset \Omega$ two open, bounded sets, $s \in (0,1)$, $p \in [2,\infty)$. Then there exists an $\eps_0 > 0$ so that for any $\eps \in (0,\eps_0)$,
\[
 [u]_{W^{s+\eps,p}(\Omega)}^{p-1} \aleq 
 [u]_{W^{s,p}(\Omega)}^{p-1}  + \|\plaps{s}{p}{\Omega} u \|_{(W_0^{s-\eps (p-1),p}(\Omega))^\ast}.
\]
\end{lemma}

\begin{proof}
We can find finitely many balls $(B_k)_{k =1}^K \subset \Omega$ so that $\bigcup_{k=1}^N B_k \supset \Omega_1$. We denote with $10B_k$ the concentric balls with ten times the radius, and may assume $\bigcup_{k=1}^N 10B_k \subset \Omega$.

Denote
\[
 \Gamma_{s} := [u]_{W^{s,p}(\Omega)}^{p} ,\quad  \Gamma_{s+\eps} := [u]_{W^{s+\eps,p}(\Omega)}^{p}.
\]
We then have
\[
 \Gamma_{s+\eps} \aleq \sum_{k=1}^K [u]_{W^{s+\eps,p}(2B_k)}^{p} + \sum_{k=1}^K \int_{\Omega \backslash 2B_k} \int_{B_k} \frac{|u(x)-u(y)|^p}{|x-y|^{n+(s+\eps)p}}\ dx\ dy.
\]
As for the second term, because of the disjoint support of the integrals we find
\[
 \int_{\Omega \backslash 2B_k} \int_{B_k} \frac{|u(x)-u(y)|^p}{|x-y|^{n+(s+\eps)p}}\ dx\ dy \aleq (\diam B_k)^{-\eps p}\ \Gamma_s.
\]
That is
\[
 \Gamma_{s+\eps} \aleq \sum_{k=1}^K [u]_{W^{s+\eps,p}(2B_k)}^{p} + \Gamma_s.
\]
With Lemma~\ref{la:lhs} and Poincar\'e inequality, Proposition~\ref{pr:poinc1}, for any $\delta > 0$,
\[
\begin{split}
 &\Gamma_{s+\eps} \aleq \delta^p \Gamma_{s+\eps} + C_\delta \Gamma_s + \sum_{k=1}^K \delta^{-p'} \brac{\sup_{\varphi} \plaps{s+\eps}{p}{8B_k}u[\varphi]}^{\frac{p}{p-1}}\\
\end{split}
\]
where the supremum is over all $\varphi \in C_c^\infty(4B_k)$ and $[\varphi]_{W^{s+\eps,p}(\R^n)} \leq 1$. Here we also used that $\bigcup_{k=1}^K 8B_k$ covers no more than $\Omega$. Choosing $\delta$ sufficiently small, we can estimate $\Gamma_{s+\eps}$ by
\[
  \Gamma_s 
+ \sum_{k=1}^K \brac{\sup \left \{ |\plaps{s+\eps}{p}{8B_k} u[\varphi]|: \ \varphi \in C_c^\infty(4B_k), [\varphi]_{W^{s+\eps,p}(\R^n)} \leq 1  \right \} }^{\frac{p}{p-1}}.
\]
With Theorem~\ref{th:epsest} this can be estimated by
\[
 \Gamma_s + \eps^{\frac{p}{p-1}} \Gamma_{s+\eps}
\]
\[ 
+ \sum_{k=1}^K \brac{\sup \left \{ |\plaps{s}{p}{8B_k} u[\laps{\eps p}\varphi]|: \ \varphi \in C_c^\infty(4B_k), [\varphi]_{W^{s+\eps,p}(\R^n)} \leq 1  \right \} }^{\frac{p}{p-1}}.
\]
If $\eps \in [0,\eps_0)$ for $\eps_0$ small enough, we can again absorb $\Gamma_{s+\eps}$. The estimate for $\Gamma_{s+\eps}$ becomes
\[
 \Gamma_s 
+ \sum_{k=1}^K \brac{\sup \left \{ |\plaps{s}{p}{8B_k} u[\laps{\eps p}\varphi]|: \ \varphi \in C_c^\infty(4B_k), [\varphi]_{W^{s+\eps,p}(\R^n)} \leq 1  \right \} }^{\frac{p}{p-1}}.
\]
Next, we need to transform $\laps{\eps p}\varphi$ into a feasible testfunction, and denoting the usual cutoff function with $\eta_{6B_k} \in C_c^\infty(6B_{k})$, $\eta_{6B_k} \equiv 1$ in $5B_{k}$
\[
 \laps{\eps p}\varphi =: \psi + (1-\eta_{6B_k})\laps{\eps p}\varphi
\]
Then $\psi \in C_c^\infty(6B_{k})$ 
\[
 [\psi]_{W^{s-\eps (p-1),p}(\Omega)} \aleq C_k [\varphi]_{W^{s+\eps,p}(\R^n)} \leq C_k.
\]
Moreover, the disjoint support of $(1-\eta_{6B_k})$ and $\varphi$ implies (see, e.g., \cite[Lemma A.1]{BPSknot12})
\[
 [(1-\eta_{6B_k})\laps{\eps p}\varphi]_{\lip} \leq C_k\ [\varphi]_{W^{s+\eps,p}(\R^n)}.
\]
Consequently, 
\[
 |\plaps{s}{p}{8B_k} u[\laps{\eps p}\varphi-\psi]| \aleq [u]_{W^{s,p}(\Omega)}^{p-1}.
\]
Hence, our estimate for $\Gamma_{s+\eps}$ now looks like
\[
 \Gamma_s 
+ \sum_{k=1}^K \brac{\sup \left \{ |\plaps{s}{p}{8B_k} u[\psi]|: \ \psi \in C_c^\infty(6B_k), [\psi]_{W^{s-\eps (p-1),p}(\R^n)} \leq 1  \right \} }^{\frac{p}{p-1}}.
\]
Finally, we need to transform the support of $\laps{s}_p$ from $8B_k$ to $\Omega$. Since $\supp \psi \subset 6B_k$, the disjoint support of the integrals gives
\[
\begin{split}
 &|\plaps{s}{p}{8B_k} u[\psi] - \plaps{s}{p}{\Omega} u[\psi]|\\
 \aleq & \int_{\Omega \backslash 8B_k} \int_{7B_k} \frac{|u(x)-u(y)|^{p-1}\ |\psi(x)-\psi(y)|}{|x-y|^{n+sp}}\ dx\ dy\\
 \leq &C_k [u]_{W^{s,p}(\Omega)}^{p-1} [\psi]_{W^{s-\eps(p-1),p}(\R^n)}.
\end{split}
 \]
This implies the final estimate of $\Gamma_{s+\eps}$ by
\[
 \Gamma_s 
+ \brac{\sup \left \{ |\plaps{s}{p}{\Omega} u[\psi]|: \ \psi \in C_c^\infty(\Omega), [\psi]_{W^{s-\eps (p-1),p}(\R^n)} \leq 1  \right \} }^{\frac{p}{p-1}}.
\]
\end{proof}

\section{Differentiability of ${p}$-harmonic maps: Proof of Theorem~\ref{th:higherreg}}\label{s:th:higherreg}
For $B \subset \R^n$, $t \in (0,1)$, we set
\[
 T_{t,B}u(z) = \int_{B} \int_B \frac{|u(x)-u(y)|^{\frac{n}{s}-2}(u(x)-u(y))\ (|x-z|^{t-n}-|y-z|^{t-n})}{|x-y|^{n+s\frac{n}{s}}}\ dx\ dy.
\]
$T_{t,B}u$ was introduced in \cite{SchikorraCPDE14} because of the following relation
\begin{equation}\label{eq:Ttbmot}
\begin{split}
 &c\int_{\R^n} T_{t,B}u(z)\ \varphi(z)\ dz\\
 = &\int_{B}\int_{B} \frac{|u(x)-u(y)|^{\frac{n}{s}-2}(u(x)-u(y))\ (\lapms{t}\varphi(x)-\lapms{t}\varphi(y))}{|x-y|^{n+s\frac{n}{s}}}\ dx\ dy.
\end{split}
\end{equation}

From \cite[in particular (3.1), Lemma 3.3, 3.4, 3.5]{SchikorraCPDE14} we have the following
\begin{theorem}\label{th:Ttest}
Let $u$ satisfy \eqref{eq:uwsns} and \eqref{eq:criticalpt} in an open set $\Omega$. Assume that on the Ball $2B$ for a small enough $\eps > 0$ (depending on $\Lambda$) \eqref{eq:smalluloc} holds. Then there is $t_0 < s$, $\sigma > 0$, so that for some $\gamma_2 > \gamma_1 \gg 1$ for any ball $B_{\gamma_2 \rho} \subset B$
\begin{equation}\label{eq:usobnormmorrey}
 [u]_{W^{s,\frac{n}{s}}(B_{\rho})} \aleq C_{\Lambda} \rho^\sigma,
\end{equation}
and
\begin{equation}\label{eq:estTbgammabrho}
\| T_{t_0,B_{\gamma_1 \rho}} u\|_{\frac{n}{n-t_0},B_\rho} \leq  C_{\Lambda} \rho^\sigma.
\end{equation}
\end{theorem}

Estimate \eqref{eq:estTbgammabrho} looks almost as if $T_{t_0,B_{\gamma_1\rho}}$ belongs locally to a Morrey space. But the domain dependence on $B_{\gamma_1 \rho}$ bars us from exploiting this. The following proposition removes the domain dependence.
\begin{proposition}\label{pr:Tt0morrey}
Under the assumptions of Theorem~\ref{th:Ttest} there exists $\gamma > 1$, $\sigma > 0$ so that
\[
\| T_{t_0,B} u\|_{\frac{n}{n-t_0},B_\rho} \leq  C_{B,\Lambda} \rho^\sigma
\]
for any ball so that $B_{\gamma \rho} \subset B$. 
\end{proposition}
\begin{proof}
Set $\kappa_1 \geq \kappa_2 \geq \kappa_3 \geq 1$ to be chosen later. Take $\gamma := 2\gamma_1$ with $\gamma_1$ from \eqref{eq:estTbgammabrho}. We will always assume $\rho < 1$.

For some $\varphi \in C_c^\infty(B_{\rho^{\kappa_1}})$, $\|\varphi\|_{\frac{n}{t_0}} \leq 1$ we have
\[
\begin{split}
 &\| T_{t_0,B} u\|_{\frac{n}{n-t_0},B_{\rho^{\kappa_1}}}\\
 \aleq & \int_{\R^n} T_{t_0,B} u\ \varphi\\
 \overset{\eqref{eq:Ttbmot}}{\aeq} & \int_{B}\int_{B} \frac{|u(x)-u(y)|^{\frac{n}{s}-2}(u(x)-u(y))\ (\lapms{t_0}\varphi(x)-\lapms{t_0}\varphi(y))}{|x-y|^{n+s\frac{n}{s}}}\ dx\ dy.\\
\end{split}
 \]
We will now use several cutoffs to slice $\varphi$ into the right form. This kind of arguments and the consequent (tedious) estimates have been used several times in work related to fractional harmonic maps, cf. e.g. \cite{DR1dSphere,DR1dMan,DndMan,BPSknot12,SchikorraCPDE14,Seps,SNHarmS10}, and we will not repeat them in detail. We will also assume that $\kappa_1 > \kappa_2 > \kappa_3$. If they are equal, to keep the ``disjoint support estimates'' working one needs to use cutoff functions on twice, four times etc. of the Balls. 

For a cutoff function $\eta_{B_{\rho^{\kappa_2}}} \in C_c^\infty(B_{2\rho^{\kappa_2}})$, $\eta_{B_{\rho^{\kappa_2}}} \equiv 1$ on $B_{\rho^{\kappa_2}}$, we have
\[
 \lapms{t_0}\varphi := \psi + (1-\eta_{B_{\rho^{\kappa_2}}}) \lapms{t_0} \varphi.
\]
Note that $\psi \in C_c^\infty(B_{2\rho^{\kappa_2}})$ and\footnote{This is true if $\frac{n}{t_0} \geq 2$, since then $[f]_{W^{t_0,\frac{n}{t_0}}} \leq \|\laps{t_0} f \|_{\frac{n}{t_0}}$. If $\frac{n}{t_0} < 2$ one has to adapt the estimate, but the results remains true.}
\begin{equation}\label{eq:psiest}
 \|\laps{t_0} \psi\|_{\frac{n}{t_0}} + [\psi]_{W^{t_0,\frac{n}{t_0}}(\R^n)} \aleq \|\varphi \|_{\frac{n}{t_0}}.
\end{equation}
The disjoint support of $(1-\eta)$ and $\varphi$ ensures (see \cite[Lemma A.1]{BPSknot12}) 
\begin{equation}\label{eq:varphimpsi}
 [\lapms{t_0} \varphi - \psi]_{W^{s,\frac{n}{s}}(\R^n)} \aleq \rho^{(\kappa_1-\kappa_2)(n-t_0)}\ \|\varphi\|_{\frac{n}{t_0}}.
\end{equation}
We furthermore decompose
\[
 \laps{t_0} \psi =: \phi + (1-\eta_{B_{\rho^{\kappa_3}}}) \laps{t_0} \psi.
\]
Then $\phi \in C_c^\infty(B_{2\rho^{\kappa_3}})$ and 
\begin{equation}\label{eq:phiest}
 \|\phi\|_{\frac{n}{t_0}} \aleq \|\varphi\|_{\frac{n}{t_0}},
\end{equation}
\begin{equation}\label{eq:psimphi}
 \|\nabla (\psi-\lapms{t_0}\phi)\|_{\infty} \aleq \rho^{-\kappa_3+(\kappa_2-\kappa_3)n}\ \|\varphi\|_{\frac{n}{t_0}}.
\end{equation}

Again with \eqref{eq:Ttbmot}, we then have
\[
\| T_{t_0,B} u\|_{\frac{n}{n-t_0},B_\rho} \aleq |I| + |II| + |III| + |IV|
\]
where
\[
 I := \int T_{t_0,B_{\gamma \rho}}u\ \phi
\]
\[
 II := \int_{B_{\gamma \rho}}\int_{B_{\gamma \rho}} \frac{|u(x)-u(y)|^{\frac{n}{s}-2}(u(x)-u(y))\ ((\psi-\lapms{t_0}\phi)(x) - (\psi-\lapms{t_0}\phi)(y))}{|x-y|^{n+s\frac{n}{s}}}\ dx\ dy
\]
\[
 III := \int_{B \backslash B_{\gamma \rho}}\int_{B_{2\rho^{\kappa_2}}} \frac{|u(x)-u(y)|^{\frac{n}{s}-2}(u(x)-u(y))\ (\psi(x) - \psi(y))}{|x-y|^{n+s\frac{n}{s}}}\ dx\ dy
\]
and
\[
 IV := \int_{B }\int_{B } \frac{|u(x)-u(y)|^{\frac{n}{s}-2}(u(x)-u(y))\ ((\lapms{t_0} \varphi-\psi)(x) - (\lapms{t_0} \varphi-\psi)(y))}{|x-y|^{n+s\frac{n}{s}}}\ dx\ dy
\]

With \eqref{eq:phiest}, $\supp \phi \subset B_{2\rho^{\kappa_3}} \subset B_{2\rho}$, and \eqref{eq:estTbgammabrho},
\[
|I| \aleq \rho^{\sigma}.
\]
With \eqref{eq:usobnormmorrey}, \eqref{eq:psimphi} (for $\rho$ small enough),
\[
 |II| \aleq [u]_{W^{s,\frac{n}{s}}(B_{\gamma\rho})}^{\frac{n}{s}-1} [\psi-\lapms{t_0}\phi]_{W^{s,\frac{n}{s}}(B_{\gamma \rho})} \aleq \rho^{\sigma(\frac{n}{s}-1)}\ \rho^{-(\kappa_3-1)} \rho^{(\kappa_2-\kappa_3)n}.
\]
With the disjoint support of the integrals, H\"older inequality ($\frac{n}{t_0} > \frac{n}{s}$), and \eqref{eq:psiest},
\[
 |III| \aleq [u]_{W^{s,\frac{n}{s}}(B)}^{p-1}\ \rho^{t_0-s}\ \rho^{\kappa_2(s-t_0)}\ [\psi]_{W^{t_0,\frac{n}{t_0}}(B)} \aleq \rho^{(\kappa_2-1)(s-t_0)}.
\]
Lastly, with \eqref{eq:varphimpsi}
\[
|IV| \aleq [u]_{W^{s,\frac{n}{s}}(B)}^{\frac{n}{s}-1} [\lapms{t_0}\varphi-\psi]_{W^{s,\frac{n}{s}}(B)} \aleq \rho^{(\kappa_1-\kappa_2)(n-t_0)}.
\]
If we choose $ \kappa_1 = \kappa_2 = \kappa_3 = 1$, we obtain
\[
 \| T_{t_0,B} u\|_{\frac{n}{n-t_0},B_{\rho}} \aleq 1,
\]
whenever $B_{2\gamma \rho} \subset B$, In particular
\begin{equation}\label{eq:TtlambdaBest}
 \| T_{t_0,B} u\|_{\frac{n}{n-t_0},\frac{1}{2\gamma}B} \aleq 1.
\end{equation}
On the other hand, we may take
\[
 \kappa_1 > \kappa_2 > \kappa_3 = 1.
\]
Then we have shown that
\[
 \| T_{t_0,B} u\|_{\frac{n}{n-t_0},B_{\rho^{\kappa_1}}} \aleq \rho^{\tilde{\sigma}},
\]
which holds whenever  $B_{\gamma \rho} \subset B$.
Equivalently, for an even smaller $\tilde{\sigma}$,
\[
 \| T_{t_0,B} u\|_{\frac{n}{n-t_0},B_{\rho}} \aleq \rho^{\tilde{\sigma}},
\]
which holds whenever $B_{\gamma \rho^{\frac{1}{\kappa_1}}} \subset B$. With \eqref{eq:TtlambdaBest} this estimate also holds whenever $B_{2\gamma \rho} \subset B$, with a constant depending on the radius of $B$.
\end{proof}

In \cite{SchikorraCPDE14} it is shown that for $t_1 > t_0$, $T_{t_1,B} u = \lapms{t_1-t_0} T_{t_0,B} u$. Since according to Proposition~\ref{pr:Tt0morrey} $T_{t_0,B} u$ belongs to a Morrey space, we can apply Adams estimates on Riesz potential acting on Morrey spaces \cite[Theorem 3.1 and Corollary after Proposition 3.4]{Adams75} and obtain an increased integrability estimate for $T_{t_1,B} u$.
\begin{proposition}
Under the assumptions of Theorem~\ref{th:Ttest} there are $\gamma > 1$, $t_0 < t_1 < s$, and $p_1 > \frac{n}{n-t_1}$ so that
\[
\| T_{t_1,B} u\|_{p_1,B_\rho} \leq  C_{\Lambda} \rho^\sigma
\]
for any ball so that $B_{\gamma \rho} \subset B$.
\end{proposition}

Now we exploit \eqref{eq:Ttbmot}: For any $\varphi \in C_c^\infty(\R^n)$
\[
 \plaps{s}{\frac{n}{s}}{B} u[\varphi] = \int_{\R^n} T_{t_1,B} u\ \laps{t_1} \varphi.
\]
Let $\varphi \in C_c^\infty(B_{\frac{1}{4}\rho})$ for $B_{\gamma \rho} \subset B$. With the usual cutoff-function $\eta \in C_c^\infty(B_{\rho})$, $\eta \equiv 1$ on $B_{\frac{1}{2}\rho}$
\[
 |\plaps{s}{\frac{n}{s}}{B} u[\varphi]| \aleq \| T_{t_1,B} u\|_{p_1,B_\rho} \|\laps{t} \varphi\|_{p_1',B_\rho} + \| T_{t_1,B} u\|_{\frac{n}{n-t},B_\rho} \|\laps{t_1} \varphi\|_{\frac{n}{t},\R^n \backslash B_{\frac{1}{2}\rho}}.
\]
By the Sobolev inequality for Gagliardo-Norms \cite[Theorem 1.6]{SchikorraCPDE14}, and the disjoint support \cite[Lemma A.1]{BPSknot12}, this implies
\[
 |\plaps{s}{\frac{n}{s}}{B} u[\varphi]| \aleq C_\Lambda [\varphi]_{W^{s+t_1-\frac{n}{p_1'},\frac{n}{s}}(\R^n)}.
\]
Since $p_1 > \frac{n}{n-t_1}$, we have $s+t_1-\frac{n}{p_1'} < s$, and the claim of Theorem~\ref{th:higherreg} follows from Theorem~\ref{th:est1} by a covering argument.\qed

\section{Compactness for $\frac{n}{s}$-harmonic maps: Proof of Theorem~\ref{th:aecompactness}}\label{s:th:aecompactness}
From the arguments in \cite[Proof of Lemma 2.3.]{DaLioBubbles} one has the following:
\begin{proposition}
For $s \in (0,1)$, $p \in (1,\infty)$ let $(u_k)_{k = 1}^\infty \in W^{s,p}(\R^n,\S^{N-1})$, $\Lambda := \sup_{k \in \N} [u_k]_{W^{s,p}(\R^n)} < \infty$ and $\eps_0 > 0$ given. 
Then up to a subsequence there is $u_\infty \in \dot{W}^{s,p}(\R^n,\S^{N-1})$ and a finite set of points $J = \{a_1,\ldots,a_l\}$ such that
\[ 
 u_k \rightharpoonup u_\infty\quad \mbox{in $W^{s,p}(\R^n,\S^{N-1})$ as $k \to \infty$},
\]
and for all $x \not \in J$ there is $r = r_x > 0$ so that
\[
 \limsup_{k \to \infty} [u_k]_{W^{s,p}(B_r(x))} < \eps_0.
\]
\end{proposition}

This, Theorem~\ref{th:higherreg} and the compactness of the embedding $W^{s+\delta,\frac{n}{s}}(B_r(x)) \hookrightarrow W^{s,\frac{n}{s}}(B_r(x))$ immediately implies that
\[
 u_k \xrightarrow{k \to \infty} u_\infty \quad \mbox{in $W^{s,\frac{n}{s}}_{loc}(\R^n \backslash J)$}.
\]
\qed.

\begin{appendix}
\section{Useful Tools}
The following Lemma is used to restrict the fractional $p$-Laplacian to smaller sets.
 \begin{lemma}[Localization Lemma]\label{la:extension}
Let $\Omega_1 \Subset \Omega_2 \Subset \Omega_3 \Subset \Omega \subset \R^n$ be open sets so that $\dist(\Omega_1,\Omega_2^c), \dist(\Omega_2,\Omega_3^c), \dist(\Omega_3,\Omega^c) > 0$. Let $s \in (0,1)$, $p \in [2,\infty)$.

For any $u \in W^{s,p}(\Omega)$ there exists $\tilde{u} \in W^{s,p}(\R^n)$ so that
\begin{enumerate}
 \item $\tilde{u} - u \equiv const$ in $\Omega_1$
 \item $\supp \tilde{u} \subset \Omega_2$
 \item $[\tilde{u}]_{W^{s,p}(\R^n)} \aleq\ [u]_{W^{s,p}(\Omega)}$
 \item For any $t \in (2s-1,s)$,
       \[
        \|\plaps{s}{p}{\Omega_3} \tilde{u}\|_{(W^{t,p}_0(\Omega_3))^\ast}  \aleq \|\plaps{s}{p}{\Omega} u\|_{(W^{t,p}_0(\Omega))^\ast} +  [u]_{W^{s,p}(\Omega)}^{p-1}.
       \]
\end{enumerate}
The constants are uniform in $u$ and depend only on $s,t,p$ and the sets $\Omega_1$, $\Omega_2$, $\Omega_3$, and $\Omega$.
\end{lemma}
\begin{proof}
Let $\Omega_1 \Subset \Omega$, let $\eta \equiv \eta_{\Omega_1} \in C_c^\infty(\Omega_2)$, $\eta_{\Omega_1} \equiv 1$ on $\Omega_1$. We set
\[
 \tilde{u} := \eta_{\Omega_1}(u-(u)_{\Omega_1}).
\]
Clearly $\tilde{u}$ satisfies property (1) and (2). We have property (3), too:
\[
 [\tilde{u}]_{W^{s,p}(\R^n)} \aleq [u]_{W^{s,p}(\Omega)}.
\]
We write
\[
 \tilde{u}(x)-\tilde{u}(y) = \underbrace{\eta(x) (u(x)-u(y))}_{a(x,y)} + \underbrace{(\eta(x)-\eta(y)) (u(y)-(u)_{\Omega_1})}_{b(x,y)}.
\]
Setting
\[
 T(a) := |a|^{p-2} a,
\]
observe that
\[
 |T(a+b) - T(a)| \aleq  |b| \brac{|a|^{p-2} + |b|^{p-2}}.
\]
Also note that
\[
 T(a(x,y)) = \eta^{p-1}(x) |u(x)-u(y)|^{p-2}(u(x)-u(y))
\]
We thus have for any $\varphi \in C_c^\infty(\Omega_3)$,
\[
\begin{split}
&\plaps{s}{p}{\Omega}\tilde{u}[\varphi]\\
 =&\int_{\Omega} \int_{\Omega} \frac{|\tilde{u}(x)-\tilde{u}(y)|^{p-2}(\tilde{u}(x)-\tilde{u}(y))\ (\varphi(x)-\varphi(y))}{|x-y|^{n+sp}}\ dx\ dy\\
 = &\int_{\Omega} \int_{\Omega} \frac{|u(x)-u(y)|^{p-2}(u(x)-u(y))\ \eta^{p-1}(x)\ (\varphi(x)-\varphi(y))}{|x-y|^{n+sp}}\ dx\ dy\\
 &+\int_{\Omega} \int_{\Omega} \frac{(T(a+b) - T(a))\ (\varphi(x)-\varphi(y))}{|x-y|^{n+sp}}\ dx\ dy\\
 = &\int_{\Omega} \int_{\Omega} \frac{|u(x)-u(y)|^{p-2}(u(x)-u(y))\ ( \eta^{p-1}(x) \varphi(x)- \eta^{p-1}(y) \varphi(y))}{|x-y|^{n+sp}}\ dx\ dy\\
  &-\int_{\Omega} \int_{\Omega} \frac{|u(x)-u(y)|^{p-2}(u(x)-u(y))\ (\eta^{p-1}(x)-\eta^{p-1}(y))\varphi(y)}{|x-y|^{n+sp}}\ dx\ dy\\
 &+\int_{\Omega} \int_{\Omega} \frac{(T(a+b) - T(a))\ (\varphi(x)-\varphi(y))}{|x-y|^{n+sp}}\ dx\ dy\\
 = &\plaps{s}{p}{\Omega}u[\eta^{p-1}\ \varphi]\\
  &-\int_{\Omega} \int_{\Omega} \frac{|u(x)-u(y)|^{p-2}(u(x)-u(y))\ (\eta^{p-1}(x)-\eta^{p-1}(y))\varphi(y)}{|x-y|^{n+sp}}\ dx\ dy\\
 &+\int_{\Omega} \int_{\Omega} \frac{(T(a+b) - T(a))\ (\varphi(x)-\varphi(y))}{|x-y|^{n+sp}}\ dx\ dy.
 \end{split}
 \]
So we have that
\[
\begin{split}
&| \plaps{s}{p}{\Omega} \tilde{u}[\varphi]|\\
\aleq & \|\plaps{s}{p}{\Omega}u\|_{(W^{t,p}_0(\Omega))^\ast}\ [\eta^{p-1} \varphi ]_{W^{t,p}(\Omega)}\\
&  + \int_{\Omega} \int_{\Omega} \frac{|u(x)-u(y)|^{p-1}\ |\eta^{p-1}(x)-\eta^{p-1}(y)|\ |\varphi(y)|}{|x-y|^{n+sp}}\ dx\ dy\\
 &+\int_{\Omega} \int_{\Omega} \frac{|\eta(x)-\eta(y)|\ |u(y)-(u)_{\Omega_1}|\ \eta(x)^{p-2}\ |u(x)-u(y)|^{p-2}\ |\varphi(x)-\varphi(y)|}{|x-y|^{n+sp}}\ dx\ dy\\
 &+\int_{\Omega} \int_{\Omega} \frac{|\eta(x)-\eta(y)|^{p-1}\ |u(y)-(u)_{\Omega_1}|^{p-1}\ |\varphi(x)-\varphi(y)|}{|x-y|^{n+sp}}\ dx\ dy.
 \end{split}
\]
That is for any $t < s$
\[
\begin{split}
&| \plaps{s}{p}{\Omega} \tilde{u}[\varphi]|\\
\aleq &  \|\plaps{s}{p}{\Omega}u\|_{(W^{t,p}_0(\Omega))^\ast}\ [\eta^{p-1} \varphi ]_{W^{t,p}(\Omega)}\\
&  + [u]_{W^{s,p}(\Omega)}^{p-1}\ \brac{\int_{\Omega} \int_{\Omega} \frac{|\eta^{p-1}(x)-\eta^{p-1}(y)|^p\ |\varphi(y)|^p}{|x-y|^{n+sp}}\ dx\ dy}^{\frac{1}{p}}\\
 &+[\varphi]_{W^{t,p}(\Omega)}\ [u]_{W^{s,p}(\Omega)}^{p-2} \brac{\int_{\Omega} \int_{\Omega_2} \frac{|\eta(x)-\eta(y)|^{p}\ |u(y)-(u)_{\Omega_1}|^p}{|x-y|^{n+(2s-t)p}}\ dx\ dy}^{\frac{1}{p}}\\
 &+[\varphi]_{W^{t,p}(\Omega)}  \brac{\int_{\Omega} \int_{\Omega_2} \frac{|\eta(x)-\eta(y)|^{p}\ |u(y)-(u)_{\Omega_1}|^p}{|x-y|^{n+(2s-t)p}}\ dx\ dy}^{\frac{p-1}{p}}.
 \end{split}
\]
Since $\eta$ is bounded and Lipschitz, $\supp \eta \subset \Omega_2$, and $\varphi \in C_c^\infty(\Omega_3)$ we have that
\[
 [\eta^{p-1} \varphi]_{W^{t,p}(\Omega)} \aleq [\varphi]_{W^{t,p}(\R^n)}.
\]
Also, choosing some bounded $\Omega_4 \Subset \Omega$ so that $\Omega_3 \Subset \Omega_4$,
\[
\begin{split}
  &\int_{\Omega} \int_{\Omega} \frac{|\eta^{p-1}(x)-\eta^{p-1}(y)|^p\ |\varphi(y)|^p}{|x-y|^{n+sp}}\ dx\ dy\\
\aleq  &\int_{\Omega_3} \int_{\Omega_4} |x-y|^{(1-s)p-n}\ \ dx\ |\varphi(y)|^p dy\\
&+\int_{\Omega_3} \int_{\R^n\backslash \Omega_4} |x-y|^{-n-sp}\ \ dx\ |\varphi(y)|^p dy\\
\aleq& \|\varphi\|_{p}^p \aleq [\varphi]_{W^{t,p}(\R^n)}^p.
\end{split}
  \]
Finally, using Lipschitz continuity of $\eta$ and that $2s-1 < t < s$
\[
\begin{split}
 &\int_{\Omega} \int_{\Omega_2} \frac{|\eta(x)-\eta(y)|^{p}\ |u(y)-(u)_{\Omega_1}|^p}{|x-y|^{n+(2s-t)p}}\ dx\ dy\\
\aleq &\int_{\Omega_3} |u(y)-(u)_{\Omega_1}|^p \int_{\Omega_2} |x-y|^{-n+(t+1-2s)p}\ dx\ dy\\
&+\int_{\Omega \backslash \Omega_3} |u(y)-(u)_{\Omega_1}|^p \int_{\Omega_2} \frac{1 }{|x-y|^{n+sp}}\ dx\ dy\\
\aleq &\int_{\Omega_1} \int_{\Omega_3} |u(y)-u(z)|^p\ dy\ dz\\
&+\int_{\Omega_1} \int_{\Omega \backslash \Omega_3} |u(y)-u(z)|^p \int_{\Omega_2} \frac{1 }{|x-y|^{n+sp}}\ dx\ dy\ dz\\
\end{split}
 \]
Note that for $x,z \in \Omega_2$ and $y \in \Omega_3^c$ we have that $|x-y| \aeq |y-z|$, and since $\Omega_1,\Omega_2,\Omega_3$ are bounded we then have
\[
\int_{\Omega} \int_{\Omega_2} \frac{|\eta(x)-\eta(y)|^{p}\ |u(y)-(u)_{\Omega_1}|^p}{|x-y|^{n+(2s-t)p}}\ dx\ dy \aleq [u]_{W^{s,p}(\Omega)}
\]
Thus we have shown that for any $\varphi \in C_c^\infty(\Omega_3)$,
\[
| \plaps{s}{p}{\Omega} \tilde{u}[\varphi]| \aleq  \brac{\|\plaps{s}{p}{\Omega}u\|_{(W^{t,p}_0(\Omega))^\ast} +[u]_{W^{s,p}(\Omega)}^{p-1}}\ [\varphi ]_{W^{t,p}(\R^n)}.
\]
Since moreover, $\supp \tilde{u} \subset \Omega_2$, for any $\varphi \in C_c^\infty(\Omega_3)$,
\[
| \plaps{s}{p}{\Omega_3} \tilde{u}[\varphi]| \aleq | \plaps{s}{p}{\Omega} \tilde{u}[\varphi]| + [u]_{W^{s,p}(\Omega)}^{p-1}\ [\varphi ]_{W^{t,p}(\R^n)},
\]
we get the claim.
\end{proof}

The next Lemma estimates the $W^{s,p}$-norm in terms of the fractional $p$-Laplacian. 
\begin{lemma}\label{la:lhs}
Let $B \subset \R^n$ be a ball and $4B$ the concentric ball with four times the radius. Then for any $\delta > 0$, $[u]_{W^{s,p}(B)}^{p} $ can be estimated by
\[
\begin{split}
& \delta^p [u]_{W^{s,p}(4B)}^{p}\\
&+ \frac{C}{\delta^{p'}} \brac{\sup_{\varphi} \int_{4B}\int_{4B} \frac{|u(x)-u(y)|^{p-2}(u(x)-u(y))\ (\varphi(x)-\varphi(y))}{|x-y|^{n+sp}}\ dx \ dy}^{\frac{p}{p-1}}\\
&+ \frac{C}{\delta^{p'}}\ \diam(B)^{-sp}\ \int \limits_{4B}  |u(x)-(u)_B|^p\ dx
\end{split}
\]
where the supremum is over all $\varphi \in C_c^\infty(2B)$ and $[\varphi]_{W^{s,p}(\R^n)} \leq 1$.
\end{lemma}
\begin{proof}
Let $\eta \in C_c^\infty(2B)$, $\eta \equiv 1$ in $B$ be the usual cutoff function in $2B$.
\[
 \psi(x) := \eta(x) (u(x)-(u)_{B}), \quad \mbox{ and } \quad 
 \varphi(x) := \eta^2(x) (u(x)-(u)_{B}).
\]
Then,
\begin{equation}\label{eq:psivarphies}
 [\psi]_{W^{s,p}(\R^n)} + [\varphi]_{W^{s,p}(\R^n)} \aleq [u]_{W^{s,p}(2B)}.
\end{equation}

We have
\[
[u]_{W^{s,p}(B)}^{p} \leq \int \limits_{4B}\int \limits_{4B} \frac{|u(x)-u(y)|^{p-2} (\psi(x)-\psi(y))\ (\psi(x)-\psi(y))}{|x-y|^{n+sp}} dx\ dy
\]
Now we observe
\[
 \begin{split}
 (\psi(x) - \psi(y))^2
 =&(\psi(x) - \psi(y))  (\eta(x) - \eta(y))(u(x)-(u)_{B})\\
 & + \psi(x)(\eta(y)-\eta(x))\ (u(x)-u(y))\\
 & + (\varphi (x) - \varphi(y))(u(x)-u(y)).
 \end{split}
 \]

That is,
\[
 [u]_{W^{s,p}(B)}^{p} \aleq I + II + III,
\]
with
\[
 I :=\int \limits_{4B}\int \limits_{4B} \frac{|u(x)-u(y)|^{p-2} (u(x)-u(y))(\varphi(x)-\varphi(y))}{|x-y|^{n+sp}} dx\ dy,
 \]
\[
 II := \int \limits_{4B}\int \limits_{4B} \frac{|u(x)-u(y)|^{p-2} |\eta(x)-\eta(y)|\ |\psi(x)-\psi(y)|}{|x-y|^{n+sp}}\ |u(x)-(u)_B|\ dx\ dy,
 \]
\[ 
III := \int \limits_{4B}\int \limits_{4B} \frac{|u(x)-u(y)|^{p-1} |\eta(x)-\eta(y)|}{|x-y|^{n+sp}}\ |\psi(x)|  dx\ dy.
 \]
With \eqref{eq:psivarphies},
\[
 I \leq [u]_{W^{s,p}(4B)}\ \sup_{[\varphi]_{W^{s,p}(\R^n)} \leq 1} \int\limits_{4B}\int\limits_{4B} \frac{|u(x)-u(y)|^{p-2}(u(x)-u(y))\ (\varphi(x)-\varphi(y))}{|x-y|^{n+sp}}\ dx \ dy.
\]
As for $II$,
\[
II \aleq \|\nabla \eta\|_{\infty}\ \int \limits_{4B}\int \limits_{4B} \frac{|u(x)-u(y)|^{p-2} |\psi(x)-\psi(y)|\ |u(x)-(u)_B| }{|x-y|^{n+sp-1}}\ \ dx\ dy.
\]
For any $t_2 > 0$ so that $ t_2 = 1-s$, we have with H\"older's inequality 
\[
\begin{split}
II  \aleq &\|\nabla \eta\|_{\infty}\ \int \limits_{4B}\int \limits_{4B} \frac{|u(x)-u(y)|^{p-2} |\psi(x)-\psi(y)|\ |u(x)-(u)_B| }{|x-y|^{n+s(p-2) + s - t_2 }}\ \ dx\ dy\\
\aleq &\ \diam(B)^{-1}\ [u]_{W^{s,p}(4B)}^{p-2}\ [\psi]_{W^{s,p}(4B)}\  \brac{\int \limits_{4B}\int \limits_{4B}  \frac{ |u(x)-(u)_B|^p }{|x-y|^{n- t_2p}}\ \ dx\ dy}^{\frac{1}{p}}.
\end{split}
\]
Since $t_2 > 0$,
\[
 \int \limits_{4B}\int \limits_{4B}  \frac{ |u(x)-(u)_B|^p }{|x-y|^{n- t_2p}}\ \ dx\ dy \aleq (\diam B) ^{t_2p}\ \int \limits_{4B}  |u(x)-(u)_B|^p\ dx
\]
So using again \eqref{eq:psivarphies}, we arrive at
\[
 II \aleq \diam(B)^{-s}\ [u]_{W^{s,p}(4B)}^{p-1}\ \brac{\int \limits_{4B}  |u(x)-(u)_B|^p\ dx}^{\frac{1}{p}}.
\]
$III$ can be estimated the same way as $II$, and we have the following estimate for $[u]_{W^{s,p}(B)}^{p}$
\[
\begin{split}
&[u]_{W^{s,p}(4B)}\ \sup_{\varphi} \int\limits_{4B}\int\limits_{4B} \frac{|u(x)-u(y)|^{p-2}(u(x)-u(y))\ (\varphi(x)-\varphi(y))}{|x-y|^{n+sp}}\ dx \ dy\\
&+[u]_{W^{s,p}(4B)}^{p-1}\ \diam(B)^{-s}\ \brac{\int \limits_{4B}  |u(x)-(u)_B|^p\ dx}^{\frac{1}{p}}
\end{split}
\]
We conclude with Young's inequality.
\end{proof}

The next Proposition follows immediately from Jensen's inequality and the definition of $[u]_{W^{t,p}(\lambda B)}^p$.

\begin{proposition}[A Poincar\'e type inequality]\label{pr:poinc1}
Let $B$ be a ball and for $\lambda \geq 1$ let $\lambda B$ be the concentric ball with $\lambda$ times the radius. Then for any $t \in (0,1)$, $p \in (1,\infty)$,
\[
\int \limits_{\lambda B}  |u(x)-(u)_B|^p\ dx \aleq \lambda^{n+tp} \diam(B)^{t p} \ [u]_{W^{t,p}(\lambda B)}^p.
\]
\end{proposition}

\end{appendix}

\bibliographystyle{plain}%
\bibliography{bib}%

\def\cprime{$'$}
\begin{thebibliography}{10}

\bibitem{Adams75}
D.~R. Adams.
\newblock A note on {R}iesz potentials.
\newblock {\em Duke Math. J.}, 42(4):765--778, 1975.

\bibitem{BjorlandCaffarelliFigalli12}
C.~Bjorland, L.~Caffarelli, and A.~Figalli.
\newblock Non-local gradient dependent operators.
\newblock {\em Adv. Math.}, 230(4-6):1859--1894, 2012.

\bibitem{BPSknot12}
S.~Blatt, {\relax Ph}.~Reiter, and A.~Schikorra.
\newblock Harmonic analysis meets critical knots (stationary points of the
  moebius energy are smooth).
\newblock {\em Trans.AMS (accepted)}, 2014.

\bibitem{Hitchhiker2}
C.~Bucur and E.~Valdinoci.
\newblock Nonlocal diffusion and applications.
\newblock {\em preprint, arXiv:1504.08292}, 2015.

\bibitem{DndMan}
F.~Da~Lio.
\newblock Fractional harmonic maps into manifolds in odd dimension $n > 1$.
\newblock {\em Calc. Var. PDE}, 48(3-4):421--445, 2013.

\bibitem{DaLioBubbles}
F.~Da~Lio.
\newblock Compactness and bubbles analysis for half-harmonic maps into spheres.
\newblock {\em Ann Inst. Henri Poincar\'e, Analyse non lin\`{e}aire},
  32:201--224, 2015.

\bibitem{DR1dMan}
F.~Da~Lio and T.~Rivi{\`e}re.
\newblock {Sub-criticality of non-local Schr\"odinger systems with
  antisymmetric potentials and applications to half-harmonic maps}.
\newblock {\em Advances in Mathematics}, 227(3):1300 -- 1348, 2011.

\bibitem{DR1dSphere}
F.~Da~Lio and T.~Rivi{\`e}re.
\newblock Three-term commutator estimates and the regularity of 1/2-harmonic
  maps into spheres.
\newblock {\em Analysis and PDE}, 4(1):149 -- 190, 2011.

\bibitem{CastroKuusiPalatucciLocalBehaviour}
A.~Di~Castro, T.~Kuusi, and G.~Palatucci.
\newblock Local behaviour of fractional $p$-minimizers.
\newblock {\em preprint}, 2014.

\bibitem{CastroKuusiPalatucciJFA14}
A.~Di~Castro, T.~Kuusi, and G.~Palatucci.
\newblock Nonlocal harnack inequalities.
\newblock {\em J. Funct. Anal.}, 267:1807--1836, 2014.

\bibitem{Hitchhiker}
E.~Di~Nezza, G.~Palatucci, and E.~Valdinoci.
\newblock Hitchhiker's guide to the fractional {S}obolev spaces.
\newblock {\em Bull. Sci. Math.}, 136(5):521--573, 2012.

\bibitem{GrafakosMF}
L.~Grafakos.
\newblock {\em Modern {F}ourier analysis}, volume 250 of {\em Graduate Texts in
  Mathematics}.
\newblock Springer, New York, second edition, 2009.

\bibitem{IannizzottoMosconiSquassina}
A.~Iannizzotto, S.~Mosconi, and M.~Squassina.
\newblock Global h\"older regularity for the fractional p-laplacian.
\newblock {\em arXiv:1411.2956}, 2014.

\bibitem{Iw92}
T.~Iwaniec.
\newblock {$p$}-harmonic tensors and quasiregular mappings.
\newblock {\em Ann. of Math. (2)}, 136(3):589--624, 1992.

\bibitem{IwaniecSbordone94}
T.~Iwaniec and C.~Sbordone.
\newblock weak minima of variational integrals.
\newblock {\em J. reine angew. Math.}, 454:143--161, 1994.

\bibitem{SireKuusiMingioneFracGehring}
T.~Kuusi, G.~Mingione, and Y.~Sire.
\newblock A fractional {G}ehring lemma, with applications to nonlocal
  equations.
\newblock {\em Rend. Lincei - Mat. Appl.}, 25:345--358, 2014.

\bibitem{SireKuusiMingioneMeasureData}
T.~Kuusi, G.~Mingione, and Y.~Sire.
\newblock Nonlocal equations with measure data.
\newblock {\em Comm. Mth. Phys.}, 337:1317--1368, 2015.

\bibitem{SireKuusiMingioneSelfImproving}
T.~Kuusi, G.~Mingione, and Y.~Sire.
\newblock Nonlocal self-improving properties.
\newblock {\em Analysis \& PDE}, 8:57--114, 2015.

\bibitem{Seps}
A.~Schikorra.
\newblock {epsilon-regularity for systems involving non-local, antisymmetric
  operators}.
\newblock {\em preprint}, 2012.

\bibitem{SNHarmS10}
A.~Schikorra.
\newblock {Regularity of n/2-harmonic maps into spheres}.
\newblock {\em J. Differential Equations}, 252:1862--1911, 2012.

\bibitem{SchikorraCPDE14}
A.~Schikorra.
\newblock Integro-differential harmonic maps into spheres.
\newblock {\em Comm.PDE}, 40(1):506--539, 2015.

\bibitem{Triebel1983}
H.~Triebel.
\newblock {\em Theory of function spaces}, volume~78 of {\em Monographs in
  Mathematics}.
\newblock Birkh\"auser Verlag, Basel, 1983.

\bibitem{VazquezDirichletEvolution}
J.-L. Vazquez.
\newblock The dirichlet problem for the fractional p-laplacian evolution
  equation.
\newblock {\em preprint, arXiv:1506.00210}, 2015.

\end{thebibliography}

\end{document}